\documentclass[11pt]{article}

\usepackage{amsmath}
\usepackage{amssymb}
\usepackage{amsfonts}
\usepackage{mathrsfs}
\usepackage[usenames]{color}
\usepackage[numbers]{natbib}
\usepackage{enumitem} %%pour faire des (ii)
\usepackage{ stmaryrd } %%\llbracket

\usepackage{graphicx}
\usepackage{epsfig}
\usepackage{amsthm}
\usepackage{mathtools}

\usepackage{latexsym}
\newsavebox{\toy}
\savebox{\toy}{\framebox[0.65em]{\rule{0cm}{1ex}}}
\newcommand{\QED}{\usebox{\toy}\end{demo}}

\numberwithin{equation}{section}

\newtheorem{theorem}{Theorem}[section]
\newtheorem{lemma}[theorem]{Lemma}

\newtheorem{cor}[theorem]{Corollary}
\newtheorem{rem}[theorem]{Remark}

\def\qed{\hfill\rule{.2cm}{.2cm}\par\medskip\par\relax}

\newcommand{\bd}{\begin{displaymath}}
\newcommand{\ed}{\end{displaymath}}
\newcommand{\Z}{{\mathbb{Z}}}

\newcommand{\R}{{\mathbb{R}}}
\newcommand{{\rd}}{\R^d}

\newcommand{\IP}{{\mathbb P}}
\newcommand{\E}{\mathbb E}
\newcommand{\DE}{\mathrm E}
\newcommand{\DP}{\mathrm P}

\newcommand{\8}{\infty}

\newcommand{\eu }{{\bf e_1}}

\renewcommand{\b}{\beta}

\newcommand{\rmd}{\mathrm{d}}
\newcommand{\D}{\Delta}

\newcommand{\e}{\varepsilon}

\newcommand{\dd}{\,\text{\rm d}}             % a straight d for differentials
\newcommand{\dB}{\xi}
\newcommand{\sZ}{{\mathscr Z}}
\newcommand{\vphi}{\varphi}
\newcommand{\nn}{\nonumber}

\newcommand{\fC }{{\mathfrak C}}

\newcommand{\heap}[2]{\genfrac{}{}{0pt}{}{#1}{#2}}

\newcommand{\ssup}[1] {{\scriptscriptstyle{({#1}})}}

\makeatletter
\def\section{\@startsection{section}{1}{\z@}{-3.5ex plus -1ex minus 
 -.2ex}{2.3ex plus .2ex}{\bf}}
\def\subsection{\@startsection{subsection}{2}{\z@}{-3.25ex plus -1ex minus 
 -.2ex}{1.5ex plus .2ex}{\bf}}
\makeatother

\newcommand{\cvlaw}{\stackrel{\rm{ law}}{\longrightarrow}}
\newcommand{\cvP}{\stackrel{ \mathbb P}{\longrightarrow}}
\newcommand{\eqlaw}{\stackrel{\rm{ law}}{=}}
\AtBeginDocument{% ceci supprime les numeros MR
   \def\MR#1{}  }

\begin{document}

\pagestyle{myheadings}
\markboth{FC-CC-CM}{Renormalization of $3d$-KPZ}

\title{Renormalizing the Kardar-Parisi-Zhang equation in $d\geq 3$ in weak disorder}

\author{Francis Comets$^{1,3}$, Cl\'ement Cosco$^{1}$, Chiranjib Mukherjee$^{2}$}

\maketitle

{\footnotesize 
\noindent$^{~1}$Universit\'e Paris Diderot\\
Laboratoire de Probabilit\'es, Statistique et Mod\'elisation\\ LPSM (UMR 8001 CNRS, SU, UPD)\\
B\^atiment Sophie Germain, 8 place Aur\'elie Nemours, 75013 Paris\\
\noindent {\tt comets@lpsm.paris,  ccosco@lpsm.paris}
\\

\noindent$^{~2}$University of M\"unster\\
Fachbereich Mathematik und Informatik\\
Einsteinstra\ss e 62, M\"unster, D-48149\\
\noindent{\tt chiranjib.mukherjee@uni-muenster.de}
\\

\noindent$^{~3}$ 
NYU-ECNU Institute of Mathematical Sciences at NYU Shanghai\\

\begin{abstract} 
{We study Kardar-Parisi-Zhang equation in spatial dimension 3 or larger driven by a Gaussian space-time white noise with a small convolution in space. 
When the noise intensity is small, it is known that the solutions converge to a random limit as the smoothing parameter is turned off. We identify this limit, in the case of general initial conditions ranging from flat to droplet. We provide strong approximations of the solution which obey exactly the limit law. We prove that this limit has sub-Gaussian lower tails, implying existence of all negative (and positive) moments.}
%We show that the rescaled pointwise fluctuations converge to a Gaussian limit as predicted by the Edwards-Wilkinson model.
\end{abstract}
\\[-.3cm]

}

\section{Introduction and main results.}

\subsection{KPZ equation and its regularization.}

We consider the {\it{Kardar-Parisi-Zhang}} (KPZ) equation 
written informally as
 \begin{equation} \label{eq:KPZ}
 \frac{\partial}{\partial t} h = \frac12 \D h +  \bigg[ \frac12   |\nabla h |^2 -\infty\bigg] +  \dB \qquad u\colon \rd\times \mathbb R_+\to \mathbb R 
 \end{equation}
and driven by a totally uncorrelated Gaussian space-time white noise $\dB$. %with covariance 
%$\E[\dB(s,x)\, \dB(t,y)]=\delta_0(t-s)\, \delta_0(x-y)$.  
More precisely, $\dB$ on $\R_+\times \rd$ is a family $\{\dB(\vphi)\}_{\varphi\in \mathcal S(\R_+ \times \rd)}$ of Gaussian random variables
$$\dB(\vphi)= \int_0^\8 \int_{\rd} \dd t \, \dd x \,\, \dB(t,x)\,\vphi (t,x)
$$
with mean 0 and covariance 
$$\mathbb E\big[ \dB(\varphi_1)\,\, \dB(\varphi_2)\big]= \int_0^\8 \int_{\rd} \varphi_1(t,x) \varphi_2(t,x) \dd t \dd x.$$
The equation \eqref{eq:KPZ} describes the evolution of a growing interface in $d+1$ dimension \cite{KPZ86,S16} and also appears as the scaling limit for $d=1$ of front propagation 
of the certain exclusion processes (\cite{BG97,C12}) as well as that of the free energy of the discrete directed polymer (\cite{AKQ14}). 
It should be noted that, on a rigorous level, only distribution-valued solutions are expected for \eqref{eq:KPZ}, and thus it is already ill-posed in $d=1$ stemming from
 the inherent non-linearity of the equation and the fundamental problem of squaring or multiplying random distributions. 
For $d=1$, studies related to the above equation have enjoyed a huge resurgence of interest in the last decade starting with 
the important work \cite{H13} which gave an intrinsic precise notion of a {\it{solution}} to \eqref{eq:KPZ}.

We now fix a spatial dimension $d\geq 3$. As remarked earlier, since \eqref{eq:KPZ} is a-priori ill-posed, we will study 
its regularized version
\begin{equation}\label{eq:KPZe}
 \frac{\partial}{\partial t} h_{\e} = \frac12 \D h_{\e} +  \bigg[\frac12   |\nabla h_\e |^2  - C_\e\bigg]+
 \b \e^{\frac{d-2}2}   \dB_{\e} \;,\qquad\,\,   h_{\e}(0,x) =0,
\end{equation}
which is driven by the spatially mollified noise 
\[
\xi_{\e}(t,x) = (\xi(t,\cdot)\star \phi_\e)(x)=  \int \phi_\e(x - y) \xi(t,y) \dd y. 
%\dB \left( \varphi_{\e,t,x}\right) \;\qquad \mbox{with}\,\,\, \varphi_{\e,t,x} (s,y) =  {\un}_{[0,t]}(s) \phi_\e(y - x)\;, 
\]
with  $\phi_\e(\cdot)=\e^{-d}\phi(\cdot/\e)$ being a suitable approximation of the Dirac measure $\delta_0$ and $C_\e$ being a suitable divergent (renormalization) constant. We will work with a fixed 
%mollifer
{{mollifier}}  
 $\phi: \R^d \to \R_+$  which is smooth and spherically symmetric, with $\mathrm{supp}(\phi) \subset B(0,\frac12)$ and  $\int_{\rd} \phi(x)\dd x=1$. 
 Then, $\{\xi_\e(t,x)\}$ is a centered Gaussian field with covariance
%\begin{equation}\nn
$$
%{\textcolor{red}{\mathbb E[ \xi_{\e}(t,x) \xi_{\e}(s,y) ]  = \mathbf{1}_{s=t} \, \e^{-d} V((x-y)/\e),}}
{{\mathbb E[ \xi_{\e}(t,x) \xi_{\e}(s,y) ]  = \delta({t-s}) \, \e^{-d} V\big((x-y)/\e\big),}}
$$ 
where $V=\phi \star \phi$ is a smooth function supported in $B(0,1)$. We also remark that in \eqref{eq:KPZe},  the multiplicative parameter $\b$ can be taken to be positive without loss of generality, while by rescaling, no multiplicative parameter is needed in \eqref{eq:KPZ}, see \cite{Q12}. Also in spatial dimensions $d\geq 3$, the factor $ \e^{\frac{d-2}2}$ is the correct scaling -- a small enough $\beta>0$ guarantees a non-trivial random limit of $h_\e$ as $\e\to 0 $, see the discussion in Section \ref{sec-discussion}. % \ref{rmk-KPZ-cross} (ii). 

The goal of the present article is to consider {\it{general solutions}} of \eqref{eq:KPZe}, namely the solutions of \eqref{eq:KPZe} with various initial conditions and prove that as the mollification parameter 
$\e$ is turned off, the renormalized solution of \eqref{eq:KPZe} converges to a meaningful random limit as long as $\beta$ remains small enough. 
We use Feynman-Kac representation of the solution of stochastic heat equation and results from directed polymers.  
Not only do we identify the distributional limit of $h_{\e}$, but we also provide a sequence (indexed by $\e$) of functions of the noise such that\begin{itemize}
\item it is a strong approximation of $h_{\e}$, i.e. the difference tends to 0 in norm,
\item all terms in the sequence obey the limit law.
\end{itemize}
 The above functions for the flat initial condition are defined from the martingale limit of a random polymer model taken at some rescaled, shifted and time-reversed version of the noise. The similar functions for other intial conditions
can be derived from the martingale limit taken at various version of the noise and the heat kernel. 
% We then identify this limit and show that it has sub-Gaussian 
We finally show that it has sub-Gaussian lower tails in this regime, which implies existence of all negative and positive moments of this object.  Besides new contributions, we gather and reformulate results which are atomized in the literature, often stated in a primitive form and hidden by necessary technicalities. We end up the introduction with a rather complete account on the state-of-the-art. We now turn to a more precise description of our main results. 
 
 \subsection{Main results.} 
 
 In order to state our main results, we will introduce the following notation which will be consistently used throughout the sequel. Recall the definition of the space-time white noise $\xi\in \mathcal S^\prime(\mathbb R\times \rd)$ which is a random tempered distribution (defined in all times, including negative ones), and for any $\varphi \in \mathcal S(\mathbb R\times \rd)$, $\e>0$, $t\in \mathbb R$ and $x\in \rd$, 
 \begin{equation}\nn
\dB^{\ssup{\e,t,x}}(\varphi) \stackrel{\mathrm{(def)}}{=}\,\, \e^{-\frac{d+2}{2} }\int_\R \int_{\rd} \varphi\big(\e^{-2}(t-s),  \e^{-1}(y-x)\big) \dB(s, y) d s\, d y.
\end{equation}
Equivalently, 
\begin{equation}\label{eq:dBetx}
\dB^{\ssup{\e,t,x}}(s,y)= \e^{\frac{d+2}2}\dB\bigg( \e^2\Big(\frac t{\e^2}-s\Big),\e\Big(y-\frac x \e\Big)\bigg)% \qquad\mbox{with} \,\, X=\e^{-1}x, T=\e^{-2}t.
\end{equation}
so that by invariance under space-time diffusive rescaling, time-reversal and spatially translation, $\dB^{\ssup{\e,t,x}}$ 
is itself a Gaussian white noise and possesses the same law as $\dB$. This is also the reason why we define the noise above also for negative times.
To abbreviate notation, we will also write 
\begin{equation}\label{eq:dBetx0}
\dB^{\ssup{\e,t}}=\dB^{\ssup{\e,t,0}}.
\end{equation}
We also need specify the definition(s) of the {\it{critical disorder parameter}}. Note that \eqref{eq:KPZe} is inherently non-linear. The Hopf-Cole transformation suggests that 
\begin{equation} \label{eq:Hopf-Cole}
u_\e=\exp h_\e
\end{equation} solves 
the linear multiplicative noise  stochastic heat equation (SHE) 
\begin{equation}\label{eq:SHEe}
 \frac{\partial}{\partial t} u_{\e} = \frac12 \D u_{\e} +
 \b \e^{\frac{d-2}2} u_\e \, \dB_{\e} \;,\qquad\,\,  u_{\e}(0,x) =1,
\end{equation}
provided that the stochastic integral in \eqref{eq:SHEe} is interpreted in the classical It\^o-Skorohod sense and that we choose 
\begin{equation}\label{C-eps}
C_\e= \b^2(\phi\star \phi)(0) \e^{-2}/2= \b^2 V(0) \e^{-2}/2
\end{equation} 
{  {equal to the It\^o correction below.}}
Then, the generalized Feynman-Kac  formula (\cite[Theorem 6.2.5]{K90}) provides a solution to \eqref{eq:SHEe} 
\begin{equation}\nn
u_{\e}(t,x)=E_{x} \bigg[ \exp\bigg\{\beta\e^\frac{d-2}{2} \,\int_0^t \int_{\rd} \, \phi_\e(W_{
{ t-s}}-y)  \dB(s, y)\dd s \dd y -  \frac{\beta^2\,t\,\e^{-2}} 2\,\, V(0)\bigg\}\bigg]\;.
\end{equation} 
with $E_x$ denoting expectation with respect to the law $P_x$ of a $d$-dimensional Brownian path $W=(W_s)_{s\geq 0}$ starting at $x\in \rd$, which is independent of the noise $\dB$. 
%Given \eqref{eq:dBetx}, we immediately get 
Plugging  \eqref{eq:dBetx} in the previous formula, using Brownian scaling and time-reversal, we get {  {the a.s. equality}   }

\begin{equation}\label{eq:uZ}
u_\e(t,x)= \mathscr Z_{\frac t{\e^2}} \left(\xi^{\ssup{\e,t}}; \frac x \e\right)
\end{equation}
where 
\begin{equation}\label{eq:Z}
\mathscr Z_T(x)= \mathscr Z_T(\xi;x)= E_x  \bigg[ \exp\bigg\{\beta \,\int_0^{T} \int_{\rd} \, \phi(W_{
{ s}}-y)  \dB(s, y)\,\dd s \dd y -  \frac{\beta^2\,T} 2\,\, V(0)\bigg\}\bigg],
\end{equation}
is the {\it normalized partition function} of the {\it Brownian directed polymer} is a white noise environment $\xi$, or equivalently, the total-mass of a {\it{Gaussian multiplicative chaos}} in the Wiener space (\cite[Section 4]{MSZ16}). 

It follows that there exists $\beta_c\in (0,\infty)$ and a strictly positive non-degenerate random variable $\mathscr Z_\infty(x)$ so that, a.s. as $T \to \8$, 
\begin{equation}\label{eq:dichotomy}
\mathscr Z_T(x) \to 
\begin{cases}
\mathscr Z_\infty(x) &\mbox{if}\,\, \beta\in (0,\beta_c),\\
0 & \mbox{if}\,\, \beta\in (\beta_c,\infty).
\end{cases}
\end{equation}
See \cite{MSZ16}, or \cite{C17}  for a general reference. Moreover, $(\mathscr Z_T)_{T\geq 0}$ is uniformly integrable for $\b<\b_c$, which we will always assume from now on. Now, let $\mathcal C^\alpha(\mathbb R\times \rd)$ denote the path-space of the white noise (see Appendix for a precise definition) and 
$$
\mathfrak u=\mathfrak u_{\beta,\phi}: \mathcal C^\alpha(\mathbb R\times \rd) \to (0,\infty),
$$
be any arbitrary representative of the random limit $\mathscr Z_\infty = \mathscr Z_\infty(0)$; in particular $\mathfrak u(\xi) = \mathscr Z_\infty$. Then, $\E[\mathfrak u]=1$, and throughout the sequel we will write (recall \eqref{eq:Hopf-Cole} and \eqref{eq:uZ})
\begin{equation}\label{eq:frak-u}
\mathfrak h=\log \mathfrak u\;.
\end{equation}
Since $\mathfrak u$ is non constant with $\E \mathfrak u=1$, we have $\E \mathfrak h <0$.

Finally, we also define another critical disorder parameter:
$$
\b_{L^2}=\sup\left\{ \b>0: E_0\bigg[e^{\beta^2\int_0^\infty V(\sqrt 2 W_s)\,\dd s}\bigg]<\8\right\}
$$ 
which corresponds to the {\it{$L^2$-region of the polymer model}} (see \eqref{eq:covM}). In $d\geq 3$, it is easy to see that for $\b$ small enough,
$\sup_{x\in \rd} E_x[\beta\int_0^\infty V(W_s)\, \dd s]<1$, so that by Khas'minskii's lemma, $ E_0\big[\exp\big\{\beta\int_0^\infty V(W_s)\, \dd s \big\}\big]<\infty$, so this implies that
$\b_{L^2}> 0$. Furthermore, for $\b<\b_{L^2}$, convergence \eqref{eq:dichotomy} becomes an $L^2$-convergence, hence $0<\b_{L^2}<\b_c<\8$.

We are now ready to state our main results.

\begin{theorem}\label{th:limh}
Assume $d\geq 3$ and recall that $\mathfrak h$ is defined in \eqref{eq:frak-u}. 
\begin{itemize}
\item (Flat initial condition.) Fix $\beta\in (0,\beta_c)$ and consider the solution $h_\e$ to \eqref{eq:KPZe} with $h_\e(0,\cdot)=0$. Then,
for all $t>0, x\in \rd$, we have as $\e \to 0$, 
\begin{equation}\nn
h_\e (t,x) - {\mathfrak h}\big( \dB^{\ssup{\e,t,x}}\big) \cvP 0 \;.
\end{equation}
\item (General initial condition.) 
Fix $\beta\in (0,\beta_{L^2})$ and consider the solution $h_\e$ to \eqref{eq:KPZe} with $h_\e(0,\cdot)=h_0(\cdot)$ for some $h_0: \rd \to \R$ which is continuous and bounded from above. Then,
for all $t>0, x\in \rd$, we have as $\e \to 0$, 
\begin{equation} \label{eq:res2}
h_\e(t,x) - {\mathfrak h}(\dB^{\ssup{\e,t,x}}) - \log \overline u (t,x) \stackrel{\IP}{\longrightarrow} 0 \;,
\end{equation}
where 
\[
\partial_t \overline u=\frac 1 2 \Delta\overline u, \qquad \overline u(0,x)=\exp h_0(x).
\]
\item \label{item:3} (Droplet or narrow-wedge initial condition.) 
Fix $\beta \in (0,\beta_{L^2})$ and consider the solution $h_\e$ to \eqref{eq:KPZe} such that 
$$ \lim_{t \searrow 0} \exp h_{\e}(t, \cdot) = \delta_{x_0}(\cdot )$$
for some $x_0 \in \rd$. Then, 
for all $t>0, x\in \rd$, we have as $\e \to 0$, 
\begin{equation} \label{eq:res3}
h_\e(t,x) - {\mathfrak h}(\dB^{\ssup{\e,t,x}})  - {\mathfrak h}(  \dB_{(\e,x_0)}  ) - \log \rho (t,x-x_0) \stackrel{\IP}{\longrightarrow} 0\;,
\end{equation}
where $\rho$ is the $d$-dimensional Gaussian kernel, and 
$$\dB_{(\e,x_0)}(s,x)= \e^{\frac{d+2}{2}} \dB (\e^2 s, x_0+\e x)$$ is  a space-time Gaussian white noise.
\end{itemize}
\end{theorem}
%%%%%%%%%%%%%%%
The deterministic terms in (\ref{eq:res2}) and (\ref{eq:res3}) are logarithms of solutions to heat equation without noise, and one can see that  $ {\mathfrak h} \to 0$ in $L^2$ as $\beta \to 0$. This implies that 
$$
\lim_{\beta \to 0} \, \lim_{\e\to 0} h_\e =  \lim_{\e\to 0} \,  \lim_{\beta \to 0} \,h_\e \;.
$$
%%%%%%%%%%%%%%%%
We obtain an immediate corollary to Theorem \ref{th:limh}.
\begin{cor}\label{cor:limh}
Fix $\beta \in (0,\beta_{L^2})$ and denote by $h_\e^{\ssup{h_0}}$ the solution of \eqref{eq:KPZe} with initial condition $h_\e(0,\cdot)=h_0(\cdot)$, where $h_0$ is continuous and bounded from above. Then for any $x_0\in \rd$, 
$$
\lim_{e^{h_0}\to \delta_{x_0}} \, \lim_{\e\to 0} h_\e^{\ssup {h_0}} \ne  \lim_{\e\to 0} \,  \lim_{e^{h_0}\to \delta_{x_0}} \,h_\e^{\ssup {h_0}}
$$
\end{cor}
Our next main result is the following which provides a sub-Gaussian upper tail estimate on the limit $\mathfrak h$ defined in \eqref{eq:frak-u}. 
\begin{theorem}\label{th:negmom}
Let $d\geq 3$ and $\beta\in (0,\beta_{L^2})$. Then for any $\theta>0$, there exists a constant $C\in (0,\infty)$ so that 
$$
\mathbb P[\mathfrak h \leq -\theta] \leq C \mathrm e^{-\theta^2/2}.
$$
In particular, $\mathfrak h\in L^p(\mathbb P)$ for any $p\in \mathbb R$.
\end{theorem}
From Theorems  \ref{th:limh} and \ref{th:negmom}, we derive
% since $\E {\mathfrak h} <0$, we deduce that the  difference between the two members of the following inequality is equal to $\E \mathfrak h >0$.
\begin{cor}\label{cor:limhE}
%Denote by $h_\e^{\ssup{h_0}}$ the solution of \eqref{eq:KPZe} with initial condition $h_0$. Then 
In the hypothesis of Corollary \ref{cor:limh}, we have
for any $x_0\in \rd$, 
$$
\lim_{e^{h_0}\to \delta_{x_0}} \, \lim_{\e\to 0} \E h_\e^{\ssup {h_0}} -  \lim_{\e\to 0} \,  \lim_{e^{h_0}\to \delta_{x_0}} \,\E h_\e^{\ssup {h_0}} \;
= - \E {\mathfrak h} >0\;.
$$
\end{cor}
%To interpret this result, let us integrate the solution $h_\e^{\ssup {h_0}} $ against a smooth test function $f:  \R^d \to \R$:
%\begin{eqnarray}\nn
%\lim_{\e\to 0} \int_{\R^d} h_\e^{\ssup {h_0}} (t,x) f(x) dx &=&
%\int_{\R^d} f(x) dx \times \E \lim_{\e\to 0}  h_\e^{\ssup {h_0}} (t,0) 
%\\ \nn
%%&=&
%% \lim_{\e\to 0}  h_\e^{\ssup {h_0}} (t,) f(x) dx 
%\end{eqnarray}

\subsection{Literature remarks and discussion.}\label{sec-discussion}

In the present set up, by finding a non-trivial limit when letting  the regularization parameter vanish we have obtained a 
non-trivial renormalization of  KPZ equation \eqref{eq:KPZ}. Let us stress the main specificity of Theorem \ref{th:limh}.
The approximating sequence $( {\mathfrak h}(  \dB^{(\e,t,x)}) ; \e >0 )$ in the case of flat initial condition combines two interesting properties:\begin{itemize}
\item it is constant in law for all $(\e,t,x)$, with law given by %the law of 
{ { the one of}}
$\log \mathscr Z_\infty$;
\item it approximates $h_\e(t,x)$ in probability.
\end{itemize}
(Similar properties hold for the other initial conditions). Since it depends on $\e$, it is not a (strong) limit, but it can be used similarly. In particular, fluctuations 
can be studied as shown in \cite{CCM18}. This is quite different from using a deterministic centering, e.g., 
$\tilde h_\e(t,x)=h_\e(t,x) - \E h_\e(t,x) $. As mentioned in \cite{DGRZ18}, $\tilde h_\e$ does not converge to 0 pointwise, but it does as a distribution. 
Integrating $\tilde h_\e$ in space against test functions cause oscillations to cancel. On the contrary, in our result 
$ h_\e(t,x) - {\mathfrak h}(  \dB^{(\e,t,x)}) \to 0$ pointwise, and 
%we do not need any spatial averaging.
{     {\it we do not need any averaging in space.}}

We also emphasize that 
our results concern studying the asymptotic behavior of the solution to the non-linear equation \eqref{eq:KPZe}, and  are not restricted to  the linear multiplicative noise stochastic heat equation (see \eqref{eq:SHEe}).{     { Furthermore, the statements}} of the results concern the solution itself, without need of integrating spatially against test functions.  
%For simplicity, we  keep the same notation $h_\e$ for the solution of the  regularized KPZ equation  \eqref{eq:KPZe} regardless of the initial conditions which will be imposed. 
However, note that the limit obtained in Theorem \ref{th:limh} {\it{does}} depend on the smoothing procedure $\phi$ as well as on the disorder parameter $\beta$ 
and it is not universal (in particular, for $\b<\b_{L^2}$, the variance of $\exp(\mathfrak h)$ can be computed from  the RHS of \eqref{eq:covM} for $x=0$, and it depends on the mollification). 
Thus the present scenario lies in total contrast with the $1$-dimensional spatial case where the limit can be defined by a chaos expansion \cite{AKQ14, AKQ} (with the parameter $\b$ absorbed by scaling) or 
via the theory of regularity structures (\cite{H14}) which also produces a renormalized limit which does not depend on the mollification scheme.

 In \cite{CCM18} we have also investigated the rate of the convergence of $h_\e \to \mathfrak h$ for small enough $\beta$ and showed that $\e^{\frac{d-2}2} [h_\e(t,x)- \mathfrak h(\xi^{\ssup{\e,t,x}})] \cvlaw N(0,\sigma^2(\beta))$ for each fixed $x\in \rd$ and $t>0$. %This Gaussian limit for small disorder parameter characterizes the Edwards-Wilkinson limit for \eqref{eq:KPZe}. 
 For larger $\beta$, the so-called {\it{KPZ regime}} is expected to take place with different limits, different scaling exponents and non-Gaussian limiting distributions. In particular, 
 the variance in the above Gaussian distribution is given by 
 $$
 \sigma^2(\beta)= C_d\int_{\rd} \dd y \, V(\sqrt 2 y) E_y[\exp\{\beta^2\int_0^\infty V(\sqrt 2 W_s) \dd s\}]
 $$
  which already diverges for $\beta>\beta_{L^2}$ indicating that the amplitude of the fluctuations, or at least their distributional nature, changes at this point. However the KPZ regime is not expected  before the critical value $\b_c$. Hence this region $\b \in (\b_{L^2}, \b_c)$ remains mysterious. 
 
Finally, we remark on the correlation structure of the limit $\mathfrak u$ which were computed in \cite{CCM18}. It was shown that, for $\beta$ small enough, 
\begin{equation}\label{eq:covM}
{\rm Cov}\big( \sZ_\8(0), \sZ_\8(x)\big)=
\begin{cases}
 E_{x/\sqrt{2}} \bigg[ \mathrm e^{\b^2 \int_0^\8 V(\sqrt 2 W_s) ds} -1 \bigg] \quad \forall x\in \rd, \\
%\\ \label{eq:covM}
%& = & P_{x}(H_{\6 B(0,1)}<\8) \times E_{\sqrt 2\eu}\left[ e^{\b^2 \int_0^\8 V(\sqrt 2 W_s) ds} -1 \right] \qquad {\rm if}  \; |x| \geq 2, \\ \nn
\fC_1  \Big(\frac{1}{|x|}\Big)^{d-2}    \ \,\qquad\qquad\qquad\qquad\forall\, |x| \geq 1,
\end{cases}
\end{equation}
with $\fC_1= E_{\eu/\sqrt 2}[ \exp\{\b^2 \int_0^\8 V(\sqrt 2 W_s) ds\} -1]$. The above correlation structure also underlines that solution $u_\e(t,x)$ and $u_\e(t,y)$ becomes asymptotically independent so that the so that the spatial averages $\int_{\rd} f(x) \, u_\e(t,x) \dd x\to \int f(x) \overline u(t,x) \dd x$ become deterministic and $\overline u$ solves the unperturbed 
heat equation $\partial_t \overline u=\frac 12 \Delta \overline u$. As remarked earlier, the spatially averaged fluctuations $\e^{1-\frac d 2} \int_{\rd} f(x)[ u_\e(t,x)- \overline u(t,x)] \dd x$ were shown to converge (\cite{MU17,GRZ17,DGRZ18}) to the averages of the heat equation with additive space-time white noise with variance given by (a constant multiple of) $\sigma^2(\beta)$, which also underlines the Edwards-Wilkinson regime in weak disorder. For averaged fluctuations of similar nature in $d=2$ we refer to \cite{CD18,CSZ18,G18}.

%We also mention that the law of the random field $H=(\mathfrak h \circ \theta_x)_x$  on $\R^d$ solves the consistency equation
%
%    $$ H(x) \eqlaw \log E_x [ \Phi_s \times e^{H \circ \theta_{(s,W_s)}} ]  $$
%%see (3.2.13) in [Comets - Cosco, notes in CIRM] for the fact it is a solution. (I don't see immediately for uniqueness of the solution).
%with $\E [\mathrm e^{H(x)}]=1$ for all $x$, $\Phi$ defined in \eqref{eq:defPhiT} and $\theta$ being the canonical shift on the path space of the white noise.  %And the sum in \eqref{eq:res3} is asymptotically independent.

\section{Proof of Theorem \ref{th:limh}}

We now consider the regularized KPZ equation  \eqref{eq:SHEe} as before, but with different initial data and identify the limit of the solution up to leading order. For notational brevity, we will write 
%\begin{equation}\label{eq:defPhiT}
%\begin{aligned}
%&\Phi_T=\Phi_T(\xi;W)=\exp\bigg\{\beta\int_0^T\int_{\rd} \phi(W_s-y) \xi(s,y) \, \dd s \dd y - \frac{\beta^2 T}2 V(0)\bigg\} \qquad\mbox{where } \,\,\, V=\phi\star \phi \\
%& \mbox{so that}\quad \mathscr Z_T(x)=\mathscr Z_T(\xi;x)=E_x[\Phi_T]\qquad\mbox{and }\quad \E[\mathscr Z_T]=1.
%\end{aligned}
%\end{equation}
\begin{equation}\label{eq:defPhiT}
\Phi_T=\Phi_T(\xi;W)=\exp\bigg\{\beta\int_0^T\int_{\rd} \phi(W_s-y) \xi(s,y) \, \dd s \dd y - \frac{\beta^2 T}2 V(0)\bigg\}  
\end{equation}
where $V=\phi\star \phi$ 
so that $\mathscr Z_T(x)=\mathscr Z_T(\xi;x)=E_x[\Phi_T]$ and $\E[\mathscr Z_T]=1$.

We also remind the reader that $u_\e$ solves \eqref{eq:SHEe} and $h_\e=\exp[u_\e]$ solves \eqref{eq:KPZe} with $C_\e= \frac{\beta^2 \e^{-2}} 2 V(0)$. Finally, 
%$u_\e(t,x)\stackrel{\mathrm{(d)}}{=} \mathscr Z_{\frac t{\e^2}} \big(\xi^{\ssup{\e,t}}; \frac x \e\big)$ where 
%$$
%\xi^{\ssup{\e,t,x}}(s,y)= \e^{\frac{d+2}2}\dB\bigg( \e^2\Big(\frac t{\e^2}-s\Big),\e\Big(y-\frac x \e\Big)\bigg)\qquad \mbox{and}\,\,
%\,\xi^{\ssup{\e,t}}=\xi^{\ssup{\e,t,0}}. 
%$$
{    {recall that $u_\e(t,x)=\mathscr Z_{\frac t{\e^2}} \big(\xi^{\ssup{\e,t}}; \frac x \e\big)$ with $\xi^{\ssup{\e,t}}$ given by \eqref{eq:dBetx}
and \eqref{eq:dBetx0}. }} 
%%and $u_\e(t,x)= \mathscr Z_{\frac t{\e^2}} \big(\xi^{\ssup{\e,t}}; \frac x \e\big)$, where $u_\e$ solves \eqref{eq:SHEe}. 

\subsection{General initial condition: Proof of \eqref{eq:res2}.}

Fix continuous functions  $u_0: \rd \to (0, +\8)$ and  $h_0:\rd \to \R$ which are bounded from above, consider the solution of SHE
\begin{equation}\label{eq:SHEeg}
 \frac{\partial}{\partial t} u_{\e} = \frac12 \D u_{\e} +
 \b \e^{\frac{d-2}2} u_\e \, \dB_{\e} \;,\qquad\,\,  u_{\e}(0,x) =u_0(x)\;,
\end{equation}
or, equivalently by the relations $u_\e= \exp h_\e$ and $u_0=\exp h_0$, the solution of KPZ
\begin{equation}\label{eq:eq:KPZeg1}
 \frac{\partial}{\partial t} h_{\e} = \frac12 \D h_{\e} +  \bigg[\frac12   |\nabla h_\e |^2  - C_\e\bigg]+
 \b \e^{\frac{d-2}2}   \dB_{\e} \;,\qquad\,\,   h_{\e}(0,x) =h_0(x),
\end{equation}
As before, we have the Feynman-Kac representation
\begin{eqnarray} \label{eq:ueg}
u_\e(t,x)= E_{x/\e}  \big[  u_0(\e W_{\e^{-2}t}) \, \Phi_{\e^{-2}t}(\xi^{\ssup{\e,t}};W)\big] %\exp\bigg\{\beta \,\int_0^{T} \int_{\rd} \, \phi(W_{
%{ s}}-y)  \dB^{(\e,t)}(s, y)d s d y -  \frac{\beta^2\,T} 2\,\, V(0)\bigg\}\bigg]
%\\ 
%&=& E_0 \bigg[  u_0(x\!+\!\e W_T) \exp\bigg\{\beta \int_0^{T} \!\!\int_{\rd}  \phi(W_{
%{ s}}\!-\!y)  \dB^{(\e,t,x)}(s, y)d s d y -  \frac{\beta^2\,T} 2 V(0)\bigg\}\bigg]
\end{eqnarray} 
with  $ \dB^{(\e,t)}$ as above. %It is easier to analyse the same quantity with fixed noise $\dB$:
%%%%%%
\begin{lemma} \label{lem:cvL2}
 For $\b \in (0, \b_{L^2})$, %the critical value for $L^2$ region, 
\begin{align*}
&E_{x/\e}  \big[  u_0(\e W_{\e^{-2}t}) \Phi_{\e^{-2} t}(\xi;W)\big] - {\mathfrak u}(\dB\circ\theta_{x/\e}) \, \overline u (t,x)  \stackrel{L^2}{\longrightarrow} 0 ,
\end{align*}
where  $\theta_{x}$ denotes the canonical spatial translation in the path space $\mathscr C^\alpha$ of the white noise and $\overline u$ solves 
%$\partial_t \overline u= \frac 12 \Delta\overline u$. 
{     {$\partial_t \overline u= \frac 12 \Delta\overline u$ with $\overline u(0,\cdot)=u_0(\cdot)$.}} 
\end{lemma}
%%%
\begin{proof}Note that 
$$
E_{0}[u_0(x + \e W_{\e^{-2}t})]=E_x[u_0(W_t)]= \overline u (t,x)
$$
Then 
\begin{align} \label{eq:expectationAbove}
& \E \bigg[\bigg(E_{x/\e} \left[ u_0( \e W_{\e^{-2}t}) \Phi_{\e^{-2}t}(\xi;W)\right] - \overline u (t,x) E_{x/\e} \left[ \Phi_{\e^{-2}t}(\xi;W)\right] \bigg)^2  \bigg]
%=\E E_{\e^{-1}x}\left[ \Phi_{\e^{-2}t} \left(u_0(x+\e W_{\e^{-2}t})-\overline u(t,x)\right) \right]^2
\\ \nonumber
&= E_{0}^{\otimes 2}\bigg[e^{\beta^2 \int_0^{\e^{-2}t} V\big(W^{\ssup 1}_s-W^{\ssup 2}_s\big) \dd s} \bigg(u_0\left(x+\e W^{\ssup 1}_{\e^{-2}t}\right)- \overline u(t,x)\bigg)
 \bigg(u_0\left(x+\e W^{\ssup 2}_{\e^{-2}t}\right)- \overline u(t,x)\bigg)
\bigg].
\end{align}
Furthermore, 
\[\left(\int_0^{\e^{-2}t} V(W^{\ssup 1}_s-W^{\ssup 2}_s) \dd s,\e W^{\ssup 1}_{\e^{-2}t},\e W^{\ssup 2}_{\e^{-2}t}\right)\cvlaw \left(\int_0^{\infty} V(W^{\ssup 1}_s-W^{\ssup 2}_s) \dd s,Z_t^{(1)},Z_t^{(2)}\right),\]
where the right hand side is a triplet of three independent random variables, with $Z_t^{(1)}$ and $Z_t^{(2)}$ distributed as $W_t$. Hence, expectation \eqref{eq:expectationAbove} vanishes as $\e \to 0$, provided that $u_0$ is bounded and continuous, and because of  uniform integrability which is implied by
\begin{equation} \label{eq:Holder}
E_{0}^{\otimes 2}\bigg[\mathrm{exp}\bigg\{(1+\delta)\beta^2 \int_0^{\infty} V\left(W^{\ssup 1}_s-W^{\ssup 2}_s\right) \dd s\bigg\}\bigg] < \infty,
\end{equation}
for $\b<\b_{L^2}$ and $\delta>0$ small enough. The proof is concluded by the observation that $E_{x/\e} \left[ \Phi_{\e^{-2}t}(\xi;W)\right]-{\mathfrak u}(\dB\circ\theta_{x/\e}) \stackrel {L^2} {\to} 0$.
\end{proof}
%%%%%%%
%By \eqref{eq:ueg} we deduce
%\begin{cor} \label{cor:uetx} For $\b < \b_{L^2}$, 
%$u_\e(t,x) -  {\mathfrak u}(\dB^{(\e,t)})  \overline u (t,x) \to 0$ in $L^2$.
%\end{cor}
%%%%%%%
We now end the\\
\noindent{\bf{Proof of \eqref{eq:res2}: }}
 For $\b < \b_{L^2}$, for all $t,x$, as $\e \to 0$, we first show
\begin{equation} \label{eq:res1}
u_\e(t,x) -  {\mathfrak u}\big(\dB^{(\e,t,x)}\big)\,  \overline u (t,x) \stackrel{L^2}{\longrightarrow} 0 \;.
\end{equation}
%and
%\begin{equation} \label{eq:res2b}
%h_\e(t,x) - {\mathfrak h}\big(\dB^{(\e,t,x)}\big) - \log \overline u (t,x) \stackrel{\IP}{\longrightarrow} 0 \;.
%\end{equation}
Note that \eqref{eq:res1} follows directly 
from Lemma \ref{lem:cvL2} and \eqref{eq:ueg}. 
Then, since $\mathfrak u>0$, taking logarithm we deduce the convergence in probability \eqref{eq:res2}. \qed
\begin{rem} Recall that in \cite{MSZ16} it was shown that, for any smooth function $f$ with compact support, $\int_{\rd} u_\e(t,x) f(x) \, \dd x \to \int_{\rd} \overline u(t,x) f(x) \, \dd x$. 
Note that unlike the latter statement, no smoothing in space is needed in the present context. 
In fact, we can recover the spatially averaged statement from above. Indeed, fast decorrelation
in space of $\dB^{(\e,t,x)}$ as $\e \to 0$, ergodicity  and smoothness  justify the equivalence below: 
\begin{eqnarray}\nn
\int u_\e (t,x) f(x) dx &\stackrel{\eqref{eq:res1}}{=} & \int  {\mathfrak u}(\dB^{(\e,t,x)})  \overline u (t,x)  f(x) dx +o(1)
\\ \nn  &\sim & \E[ {\mathfrak u}(\dB^{(\e,t,x)}) ] 
\int   \overline u (t,x)  f(x) \dd x \\ \nn & = &\int   \overline u (t,x)  f(x) \dd x .
\end{eqnarray}
\end{rem} 
\qed

\subsection{Narrow-wedge initial condition: Proof of \eqref{eq:res3}.}
Fix $x_0 \in \rd$, and consider the solution of SHE
\begin{equation}\label{eq:SHEegNarrow}
 \frac{\partial}{\partial t} u_{\e} = \frac12 \D u_{\e} +
 \b \e^{\frac{d-2}2} u_\e \, \dB_{\e} \;,\qquad\,\, \lim_{t \searrow 0} u_{\e}(t, \cdot) = \delta_{x_0}(\cdot)\;,
\end{equation}
or, equivalently by the relation $u_\e= \exp h_\e$, the solution of KPZ
\begin{equation}\label{eq:eq:KPZeg}
 \frac{\partial}{\partial t} h_{\e} = \frac12 \D h_{\e} +  \bigg[\frac12   |\nabla h_\e |^2  - C_\e\bigg]+
 \b \e^{\frac{d-2}2}   \dB_{\e} \;,\quad\,\,   %h_{\e}(0,x) = ???
  \lim_{t \searrow 0} \exp h_{\e}(t, \cdot) = \delta_{x_0}(\cdot)\;.
\end{equation}
%{\bf [What is the log of a Dirac ?  Is it a Dirac ?? A Dirac in the $(\max,+)$ algebra ???]} 
%Similarly to \eqref{eq:pfsx}, 
By Feynman-Kac formula, the solution of SHE now admits a Brownian bridge representation:
\begin{align} 
\nonumber  &u_\e(t,x) \\ 
&= \rho(t,x-x_0) \,E_{0,\e^{-1}x_0}^{\e^{-2}t,\e^{-1}x}  \bigg[  \exp\bigg\{\beta \,\int_0^{\e^{-2}t}\!\! \int_{\rd} \, \phi(W_{
{ s}}-y)  \dB_{(\e,0)}(s, y)\dd s \dd y -  \frac{\beta^2\,t}{2\,\e^2}\,\, V(0)\bigg\}\bigg]\\ \label{eq:BrownianBridgeRepr}
&= \rho(t,x-x_0) E_{0,0}^{\e^{-2}t,\e^{-1}x}  \bigg[  \exp\bigg\{\beta \,\int_0^{\e^{-2}t}\!\! \int_{\rd} \, \phi(W_{
{ s}}-y)  \dB_{(\e,x_0)}(s, y)\dd s \dd y -  \frac{\beta^2\,t}{2\,\e^2}\,\, V(0)\bigg\}\bigg] 
\end{align} 
where $E_{0,x}^{t,y}$ denotes expectation with respect to a Brownian bridge starting at $x$ and conditioned to be found at $y$ at time $t$, $\rho$ is the $d$-dimensional Gaussian kernel and 
\begin{equation}\label{eq:xi:subscript}
\dB_{(\e,x_0)}(S,Y) = \e^{(d+2)/2}\dB(\e^2 S,x_0 + \e Y)
\end{equation}
% (no time reversal). 
so that we again have $\dB_{(\e,x_0)}\eqlaw \dB$. %\textcolor{blue}{Here, I chose not to reverse time, as I believe it is not necessary, for this initial condition, if we choose $x$ to be the endpoint of the Brownian bridge. Tell me if you agree. It does not change much, but it prevents us from having to time reverse two times in the following lemma and theorem, so it should be a little clearer.}
%%%%%%

The following Lemma follows the approach for proving the local limit theorem as in \cite{Si95,Va04}.
%%%%%
\begin{lemma} \label{lem:cvL2sharp}
 For $\b \in (0, \b_{L^2})$, for any $A>0$,%the critical value for $L^2$ region, as $\e \to 0$.
\begin{equation} \label{eq:LLT}
\begin{split}
\sup_{|x|\leq A}\big\| E_{0,0}^{\e^{-2}t,\e^{-1}x}[  \Phi_{\e^{-2}t}(\xi,\cdot)]
 -  {\mathfrak u}(\dB)  
 \; {\mathfrak u}\big(\dB^{(1,\e^{-2}t,\e^{-1}x)}\big)\big\|_{L^1(\mathbb P)} \to0.
\end{split}
\end{equation}

%\textcolor{blue}{CC: do we want to add some uniformity in $x$ for the error term ?} 
%$\theta_x$ is the shift in space and  $ {\mathcal S}_\e: (t,x) \mapsto (\e^{-2}t, \e^{-1}x)$ is diffusive scaling.
\end{lemma}
%%%
\begin{proof}
%$\Box$ We follow thoroughly the proof of \cite[Theorem 2.9]{Va04}, which shows the corresponding result in the case of a Brownian polymer in Poisson environment.

We will write $X=\e^{-1} x$, $T=\e^{-2}t$ and let $m=m_\e$ be a time parameter, such that $m_\e\to \infty$ and $m_\e = o({T})$, as $\e \to 0$. We use the notation:
\[\Phi_{S,T}(\xi;W) := \exp\bigg\{\beta \,\int_S^{T} \int_{\rd} \, \phi(W_{
{ s}}-y)  \dB (s, y)\dd s \dd y -  \frac{\beta^2\,(T-S)} 2\,\, V(0)\bigg\}.\]

\paragraph{Step 1:} We first want to approximate $E_{0,0}^{T,X}[\Phi_T]$ by $E_{0,0}^{T,X}\left[\Phi_{m}\Phi_{T-m,T}\right]$ in $L^2$-norm, so we compute the difference:
$$
\E\bigg[\bigg(E_{0,0}^{T,X} [\Phi_T - \Phi_{m}\Phi_{T-m,T}]\bigg)^2\bigg] = E_{0,0}^{T,0} \left[\mathrm e^{\b^2 \int_0^T V(\sqrt 2 W_s)ds} -\mathrm e^{\b^2 \int_0^{m} V(\sqrt 2 W_s)ds}\mathrm e^{\b^2 \int_{T-m}^T V(\sqrt 2 W_s)ds}\right].
$$
To show that the right hand side which goes to $0$ as $\e \to 0$, it suffices to observe that, for all $a>0$,
\[\lim_{\e\to 0} E_{0,0}^{T,0}\left[\mathrm e^{\b^2\int_0^{T} V(\sqrt 2  W_s)\dd s}\,\,\, \mathbf 1\bigg\{\int_m^{T-m} V(\sqrt 2 W_s)\,\dd s >a\bigg\}
 \right] = 0.\]
To prove this, we use H\"older's inequality similarly to \eqref{eq:Holder}, and {\textcolor{black}{apply \cite[Lemma 3.5]{CCM18} (alternatively \cite[Corollary 3.8]{Va04})}}  and transience of Brownian motion for $d\geq 3$, which implies, since $m\to\infty$, that
\[\lim_{\e\to 0} E_{0,0}^{T,0}\left[\mathrm \mathbf 1\bigg\{\int_m^{T-m} V(\sqrt 2 W_s)\,\dd s >a\bigg\}
 \right] = 0.\]
 
\paragraph{Step 2:} 
We wish to use Markov property and symmetry of the Brownian bridge to show that  $E_{0,0}^{T,X}\left[\Phi_{m}\Phi_{T-m,T}\right]$ factorizes asymptotically into the product $E_0\left[\Phi_{m} \right] E_0[\Phi_{m}(\dB^{(1,T,X)})]$, which satisfies:
\[
\sup_{x\in\mathbb{R}} \left\Vert E_0\left[\Phi_{m} \right] E_0[\Phi_{m}(\dB^{(1,T,X)})] - {\mathfrak u}(\dB)  
 \; {\mathfrak u}(\dB^{(1,\e^{-2}t,\e^{-1}x)}) \right\Vert_1 \to 0,
\]
as $\e\to 0$ for $\b<\b_{L^2}$, by Cauchy-Schwarz inequality and invariance in law of the white noise with a shift by $X$. Hence, we compute
\begin{align}
E_{0,0}^{T,X}\left[\Phi_{m}\Phi_{T-m,T}\right] & = \int_{\mathbb{R}^d} \frac{\rho_{T/2}(Y)\rho_{T/2}(X-Y)}{\rho_{T}(X)} E_{0,0}^{T/2,Y} \left[\Phi_{m}\right] E_{T/2,Y}^{T,X}  \left[\Phi_{T-m,T}\right] \dd Y.
\end{align}
After change of variable by setting $Y=\sqrt{T}y$ in the above integral and since $\rho(Ts,\sqrt{T}z) = T^{-d/2} \rho(s,z)$, observe that by Jensen's inequality and dominated convergence, we can prove that
\[
\sup_{|x|\leq A} \left\Vert E_{0,0}^{T,X}\left[\Phi_{m}\Phi_{T-m,T}\right] - E_0\left[\Phi_{m} \right] E_0[\Phi_{m}(\dB^{(1,T,X)})] \right\Vert_1 \to 0,\] 
if we can show that for all fixed $y\in\mathbb{R}^d$,
\[
\sup_{|x|\leq A} \left\Vert E_{0,0}^{T/2,\sqrt T (y-x)} \left[\Phi_{m}\right] - E_0\left[\Phi_{m} \right]\right\Vert_1 \to 0.\] 
 
To prove this, we use the density of the Brownian bridge, at truncated time horizon, with respect to the Brownian motion (\cite[Lemma 3.4]{CCM18}), to get that: 
\[E_{0,0}^{T/2,\sqrt T (y-x)} \left[\Phi_{m}\right] - E_0\left[\Phi_{m} \right] = E_0\left[\Phi_{m} \left(\frac{\rho_{T/2-m}\big(\sqrt T (y-x) - W_m\big)}{\rho_{T/2}\big(\sqrt T (y-x)\big)}-1\right)\right].\]
After rescaling, the difference inside the parenthesis goes almost surely to $0$, for $y$ fixed and uniformly in $|x|\leq A$; we conclude the proof of the lemma using H\"older's inequality.

\end{proof}

We can now conclude the 
%%%%%%%%%%%%

\noindent{\bf{Proof of \eqref{eq:res3}:}} 
 Let $h_\e$ be the narrow-wedge height function solution of of \eqref{eq:eq:KPZeg}. We need to show that for $\b < \b_{L^2}$, for all $t,x$,  as $\e \to 0$, 
$$
h_\e(t,x) - {\mathfrak h}(  \dB_{(\e, x_0)}  ) - {\mathfrak h}(\dB^{(\e,t,x)})   - \log \rho (t,x-x_0) \stackrel{\IP}{\longrightarrow} 0,
$$
with $\dB_{(\e, x_0)}$ in \eqref{eq:xi:subscript}. 
We use the representation \eqref{eq:BrownianBridgeRepr} and the property $\dB_{(\e,x_0)} \eqlaw \dB$, so that we can exchange $\dB$ with $\dB_{(\e,x_0)}$, in convergence \eqref{eq:LLT} taken with endpoint $\e^{-1}(x-x_0)$. This leads to the above convergence in probability for the logarithm, proving \eqref{eq:res3}.
\qed

\section{Proof of Theorem \ref{th:negmom}.}
%%%
We focus on showing that for $\b<\b_{L^2}$, $\log \sZ_\infty$ admits a sub-Gaussian lower tails estimate, that is, for some $C\in (0,\infty)$ and any $\theta>0$, 
\begin{equation}\label{moment:claim1}
\mathbb P[\log \sZ_\infty \leq -\theta] \leq C\mathrm e^{-\theta^2/C}.
\end{equation}

 %A proof of the corresponding result for the KPZ equation in dimension 2, which relies on the convexity of the free energy and the malliavin derivative, can be found in \cite{CSZ18}

We invoke a second moment method combined with the Talagrand's concentration inequality as in \cite{CH02} (see also \cite[Section 2.2]{M14}).  As the Cauchy-Schwarz inequality, which is a central tool in this proof is not directly available in the continuous setting, we choose to introduce a discretization of the white noise to recover it.

Consider $\mathcal R_n$ a tiling of $[0,2^n]\times[-2^n,2^n]^d$, composed of cubes of length $2^{-n}$, such that every cube of $\mathcal R_{n}$ can be divided in $2^{d+1}$ cubes of $\mathcal R_{n+1}$. We define a discrete version of $\sZ_T$ through:
\begin{equation*}\sZ_T^{(n)} = \DE \left[\exp\left\{ \b \int_0^T \int_{\mathbb{R}^d} \phi^{(n)}_W(s,y)\, \xi(s,y) \rmd s \rmd y - \frac{\b^2}{2} \left\Vert \phi_W^{(n)} \right\Vert_{L^2([0,T]\times \mathbb{R}^d)}^2 \right\}\right],
\end{equation*}
where,
\begin{itemize}
\item $\phi_W^{}(s,y) := \phi(W_s-y)$,
\item $\phi_W^{\ssup n}(s,y) = \inf_{R} \phi_W^{}$ if $(s,y)$ is in a cube $R$ of $\mathcal{R}_n$, and $0$ otherwise.
\end{itemize}
We stress that $\phi_W^{\ssup n}$ is non-decreasing with $n$ and converges almost surely to $\phi_W$.

Using the Gaussian covariance structure, we have that
\begin{equation*}
\begin{aligned}
&\E \left[\left(\sZ_T - \sZ_T^{(n)}\right)^2\right]  = \DE^{\otimes 2}\bigg[\mathrm e^{\frac{\b^2}{2} \int_{[0,T]\times \mathbb{R}^d} \phi_{W^{\ssup 1}} \phi_{W^{\ssup 2}} (s,y) \rmd s \rmd y}\bigg]\\ &\qquad-2\DE^{\otimes 2}\bigg[\mathrm e^{\frac{\b^2}{2} \int_{[0,T]\times \mathbb{R}^d} \phi^{(n)}_{W^{\ssup 1}} \phi^{}_{W^{\ssup 2}} (s,y) \rmd s \rmd y}\bigg] 
 + \DE^{\otimes 2}\bigg[\mathrm e^{\frac{\b^2}{2} \int_{[0,T]\times \mathbb{R}^d} \phi^{(n)}_{W^{\ssup 1}} \phi^{(n)}_{W^{\ssup 2}} (s,y) \rmd s \rmd y}\bigg].
\end{aligned}
\end{equation*}
For $\b < \b_{L^2}$ and since $\phi_W^{\ssup n}\leq \phi_W$, we immediatly obtain from the monotone convergence theorem that the right-hand side goes to zero in the limit $T\to \infty$ followed by $n\to\infty$. By Doob's $L^2$ inequality 
applied to the martingale $\sZ-\sZ^{\ssup n}$, this implies in particular that
\begin{equation} \label{eq:L2approxResult}
\lim_{n\to\infty}\E \left[\left(\sZ^{}_\infty - \sZ_\infty^{(n)}\right)^2\right] = 0.
\end{equation}

Since $\phi_W^{\ssup n}$ is set to be $0$ outside of $[0,2^n]\times[-2^n,2^n]^d$, the set $\mathcal{C}_n$ containing the centers of the cubes of $\mathcal{R}_n$ is finite. Hence, as each cube has volume $2^{-(d+1)n}$, we can write that
\begin{equation} \label{eq:discreteExpression}
\int_0^\infty \int_{\mathbb{R}^d} \phi^{\ssup n}_W(s,y)\, \xi(s,y) \rmd s \rmd y = \frac{1}{2^{\frac{(d+1)n}{2}}}\sum_{(i,x)\in \mathcal{C}_n} \phi_W^{\ssup n}(i,x) \, \xi_n(i,x),
\end{equation} where the $\xi_n(i,x)$ are independent centered Gaussian random variables of variance $1$. Then, we define the polymer measure $\widehat P_{\b,\xi_n}$ of renormalized partition function $\sZ_\infty^{\ssup n}$: %in a similar way as \eqref{eq:polymer_measure_def}:
\[ \widehat P_{\b,\xi_n}(\rmd W) =  \frac{1}{\sZ_\infty^{\ssup n}} \, \Phi^{\ssup n}(W)  \DP(\rmd W),\]
where we have set, using \eqref{eq:discreteExpression},
\[\Phi^{\ssup n}(W) = \exp\left\{ \frac{\b}{2^{\frac{(d+1)n}{2}}}\sum_{(i,x)\in \mathcal{C}_n} \phi_W^{\ssup n}(i,x) \, \xi_n(i,x) - \frac{\b^2}{2} \left\Vert \phi_W^{\ssup n} \right\Vert_{L^2([0,\infty]\times \mathbb{R}^d)}^2 \right\}.\]
Finally, we let $\widehat E_{\beta,\xi_n}$ denote expectation corresponding to $\widehat P_{\beta,\xi_n}$. 

Now, we can compare the free energies of two realizations of the noise $\xi_n(i,x)$ and $\xi'_n(i,x)$:
\begin{align*}
&\log \sZ_\infty^{\ssup n}(\xi_n) - \log \sZ_\infty^{\ssup n}(\xi'_n) \\ 
& = \log \widehat E_{\b,\xi'_n} \left[\exp\left\{\frac{\b}{2^{\frac{(d+1)n}{2}}}\sum_{(i,x)\in \mathcal{C}_n} \phi_W^{\ssup n}(i,x) \, \left(\xi_n(i,x)-\xi'_n(i,x)\right) \right\}  \right]\\
& \geq \b \sum_{(i,x)\in \mathcal{C}_n} {2^{-\frac{(d+1)n}{2}}}\, \widehat \DE_{\b,\xi'_n} \left[ \phi_W^{\ssup n}(i,x) \right] \left(\xi_n(i,x)-\xi'_n(i,x)\right)\nonumber\\
& \geq - \b \sqrt{ \widehat E_{\b,\xi'_n}^{\otimes 2} \left[ \int_0^\infty \int_{\mathbb{R}^d} \phi_{W^{\ssup 1}}^{(n)}\phi_{W^{\ssup 2}}^{(n)}(s,y)\rmd s \rmd y \right]} \, \rmd (\xi_n,\xi'_n),
\end{align*}
where $\rmd(\cdot,\cdot)$ denotes the euclidean distance on $\mathbb{R}^{\mathrm{Card}(\mathcal C_n)}$, and where we used Jensen's and Cauchy-Schwarz inequalities for respectively the first and second lower bounds.

Let $m$ and $C$ be two positive constants, and consider the set:
\begin{equation*} 
\mathcal A_n=\left\{\xi_n : \sZ_\infty^{(n)}(\xi_n)\geq m, \,\, \widehat E_{\b,\xi_n}^{\otimes 2} \left[ \int_0^\infty \int_{\mathbb{R}^d} \phi_{W^{\ssup 1}}^{(n)}\phi_{W^{\ssup 2}}^{(n)}(s,y)\rmd s \rmd y \right] \leq C^2\right\}.
\end{equation*}
For $\xi'_n \in \mathcal A_n$, the above computation implies that
\begin{equation}\label{eq:lbound_on_apprFreeEnergy}
\log \sZ_\infty^{\ssup n}(\xi_n) \geq \log m - \b C \rmd\left(\xi_n,\mathcal A_n\right),
\end{equation}
therefore, assuming
\begin{equation} \label{lbProbA}
\liminf_{n\to\infty} \IP(\mathcal A_n)> 0,
\end{equation}
property \eqref{moment:claim1} results from \eqref{eq:lbound_on_apprFreeEnergy}, \eqref{eq:L2approxResult} and the following Gaussian concentration inequality (Lemma 4.1 in \cite{CH02}, extracted from \cite{Ta00}):
\begin{equation} \label{eq:distanceConcentration}
\IP\left(\rmd (\xi_n,\mathcal A_n) > u + \sqrt{2\log\frac 1 {\IP(\mathcal A_n)}}\right) \leq e^{-\frac{u^2}{2}}.
\end{equation}
We now prove \eqref{lbProbA}. By convergence \eqref{eq:L2approxResult}, and since ${\mathscr Z}_\8>0$ a.s., we can find $m>0$, such that for $n$ large enough,
\begin{equation*}
\IP\left(\sZ_\infty^{(n)} > m\right) \geq \frac 1 2.
\end{equation*}
Then, for a large enough $C$,
\begin{align*}
&\IP(\mathcal A_n) 
\geq \IP\left(\sZ_\infty^{(n)}\geq m, \,\, \DE^{\otimes 2}\left[\int_0^\infty \int_{\mathbb{R}^d} \phi_{W^{(1)}}^{(n)}\phi_{W^{(2)}}^{(n)}(s,y)\rmd s \rmd y \  e^{\beta \Phi^{(n)}(W^{\ssup 1})+\beta\Phi^{(n)}(W^{\ssup 2})}\right] \leq mC^2 \right) \nonumber \\
&\geq \IP\left(\sZ_\infty^{(n)}\geq m\right) -\IP\left(\DE^{\otimes 2}\left[\int_0^\infty \int_{\mathbb{R}^d} \phi_{W^{(1)}}^{(n)}\phi_{W^{(2)}}^{(n)}(s,y)\rmd s \rmd y \  e^{\beta \Phi^{(n)}(W^{\ssup 1})+\beta\Phi^{(n)}(W^{\ssup 2})}\right] > mC^2 \right) \nonumber \\
&\geq \frac 1 2 - \frac 1 {mC^2} \DE^{\otimes 2}\bigg[\int_0^\infty\, V(W^{\ssup 1}_s- W^{\ssup 2}_s)\dd s \  \mathrm e^{\beta^2\int_0^\infty V(W^{\ssup 1}_s-W^{\ssup 2}_s)\dd s}\bigg] >0.
\end{align*}
In the above display, the first lower bound follows from the definition of $\mathcal A$, while we used $\mathbb P[A\cap B] \geq \mathbb P[A]-\mathbb P[B^c]$ in the second lower bound. The third inequality comes from Markov's inequality and the upper-bound $\phi_W^{(n)}\leq \phi_W$. Positivity of the left hand-side of the third line is assured for $C$ large enough, provided that $\b<\b_{L^2}$.

This entails that the KPZ limit $|\mathfrak h|$ (recall \eqref{eq:frak-u}) has all positive and negative moments for all $\b<\b_{L^2}$.
Indeed, letting $\log_- = \log \wedge \, 0$, the sub-Gaussian decay of the left tail of $\log \sZ_\infty$ \eqref{moment:claim1} gives $\E[\exp\{\nu\log_- \sZ_{\infty}\}] <\infty $, for all $\nu \in \mathbb{R}$. Moreover, by definition of the $L^2$ region, we have $\E[\exp\{2\log \sZ_{\infty}\}] <\infty $. Hence, $\log \sZ_\infty$ admits all positive and negative moments.
\qed 

 \begin{rem}
 After finishing the writing of the present article, we learned that another proof of the negative moments of $\mathscr Z_\infty$ has been recently proposed in \cite{HL18,DGRZ18}
using a continuous approximation of the white noise.  A proof of the corresponding result for the KPZ equation in dimension 2, which relies on the convexity of the free energy and the Malliavin derivative, can be found in \cite{CSZ18}.
\end{rem}

\appendix

\section{Appendix.}\label{sec-appendix} 

We will include some elementary facts regarding the regularity properties of space-time white noise $\xi$. For any $z,z^\prime\in \R\times \rd$, we will denote by $\|\cdot\|$ the {\it{parabolic distance}} given by 
$\|z-z^\prime\|= |t-t^\prime|^{1/2}+ \sum_{i=1}^d |x_j-x_j^\prime|$, where $z=(t,x)$ and $z^\prime=(t^\prime,x^\prime)$. Recall that the H\"older space of positive exponent $\alpha\in (0,1)$ consists of all functions $u: \R\times \rd \to \R$ such that for any compact set $K\subset \R\times \rd$, 
$$
%\sup_{z,z^\prime\in K} \frac {|u(z)- u(z^\prime)|}{\|z-z^\prime\|} 
{{\sup_{z,z^\prime\in K, z \neq z^\prime} \frac {|u(z)- u(z^\prime)|}{\|z-z^\prime\|^\alpha}}} < \infty.
$$
The corresponding H\"older (Besov) space of {\it{negative regularity}} is defined as follows. First for any $k\in \mathbb N$, let $B_k$ denote the space of all smooth functions $\varphi\colon \R\times \rd \to \R$ which are supported on the unit ball in $(\R\times \rd, \|\cdot\|)$ such that 
$$
\|\varphi \|_{B_k}\stackrel{\mathrm{(def)}}{=} \sup_{\beta: |\beta| \leq k} \, \sup_{z\in \R\times \rd} |D^\beta \varphi(z)| \leq 1.
$$
Then for any fixed $\alpha<0$, we define the space $\mathscr C^\alpha$ to be the space of all tempered distributions $\eta \in \mathcal S^\prime(\R\times \rd)$ such that for any compact set $K\subset \R\times \rd$, 
$$
\|\eta\|_{\mathscr C^\alpha(K)}\stackrel{\mathrm{(def)}}{=} \sup_{z\in K} \sup_{\heap{u\in B_k}{\lambda\in (0,1]}} \bigg| \frac{\big\langle \eta, \Theta^\lambda_z u\big\rangle}{\lambda^\alpha}\bigg| <\infty,
$$
where $k=\lceil-\alpha\rceil$ and 
$$
(\Theta^\lambda_z u)(s,y)= \lambda^{-(d+2)} u\big(\lambda^{-2}(t-s), \lambda^{-1}(y-x)\big) \qquad z=(t,x).
$$
A crucial estimate on the Besov norm $\|\cdot\|_{\mathscr C^\alpha(K)}$ is given by 
\begin{equation}\label{est:Besov}
\|\eta\|_{\mathscr C^\alpha(K)} \leq C \sup_{n\geq 0} \sup_{z\in (2^{-2n} \Z \times 2^{-n}\Z^d)\cap \widetilde K} 2^{-n\alpha}\big\langle\eta, \Theta^{2^{-n}}_z u\big\rangle,
\end{equation}
where $\widetilde K$ is also a compact set slightly larger than $K$ and $u$ is a single, well-chosen test function (which can be constructed by wavelets, see \cite{H14}). Recall that if $\xi$ is space-time white noise (i.e, $\E[\langle \xi, \varphi_1\rangle \langle \xi, \varphi_2\rangle]= \int \varphi_1(t,x) \varphi_2(t,x) \, \dd t \dd x$), then with 
\begin{equation}\label{est:xi}
\langle \xi_\lambda, \varphi\rangle = \langle \xi, \Theta^\lambda_0 \varphi\rangle\qquad\mbox{ we have }\qquad \E[\langle \xi_\lambda, \varphi\rangle^2]=\lambda^{-(d+2)} \int_{\R^{d+1}} \varphi^2(t,x) \, \dd t \dd x.
\end{equation} 
The following result, which is a consequence of Kolmogorov's lemma and \eqref{est:Besov}, then implies the desired regularity property of $\xi$.
\begin{lemma}\label{lemma:appendix}
Fix $\alpha<0$ and $p\geq 1$ and let $\eta$ be a linear map from $\mathcal S(\R\times \rd)$ to the space of random variables. Suppose there exists $C\in (0,\infty)$ such that for all $z\in \R\times \rd$ and all $u\in \mathcal S(\R\times \rd)$ with compact support in $\R\times \rd$ with $\sup_w |u(w)| \leq 1$ one has 
$$
\E[|\eta(\Theta^\lambda_z u)|^p] \leq C \lambda^{\alpha p} \qquad \forall \lambda\in (0,1].
$$
Then there exists a random distribution $\widetilde\eta$ in $\mathcal S(\R\times \rd)$ such that for all $\alpha^\prime < \alpha- \frac {(d+2)}p$ and compact set $K$,
$$
\E\big[\|\widetilde \eta \|^p_{\mathscr C^{\alpha^\prime}(K)}\big] <\infty \qquad\mbox{and}\, \, \eta(u)=\widetilde \eta(u) \quad\mbox{a.s.}
$$
\end{lemma}
\qed
Since for any $p \geq 1$, $\E[\langle \xi_\lambda,\varphi\rangle^p] \leq C_p \E[\langle\xi_\lambda,\varphi\rangle^2]^{p/2}$, then Lemma \ref{lemma:appendix} and \eqref{est:xi} imply that $\xi$ has regularity 
$\mathscr C^{-\frac d2- 1- \delta}$ for any $\delta>0$.

\vspace{3mm}

\noindent{\bf Compliance with Ethical Standards:} The authors declare that they have no conflict of
 interest.

\end{document}